\documentclass{amsart}%
\usepackage{amsfonts}
\usepackage{amsmath}
\usepackage{amssymb}
\usepackage{graphicx}%
\providecommand{\U}[1]{\protect\rule{.1in}{.1in}}

\theoremstyle{plain}
\newtheorem{theorem}{Theorem}

\newtheorem{lemma}{Lemma}

\theoremstyle{definition}

\newtheorem{definition}{Definition}

\newtheorem{remark}{Remark}

\numberwithin{equation}{section}

\begin{document}
\title{Generalizing Krawtchouk polynomials using Hadamard matrices}

\author{Peter S Chami }
\author{Bernd Sing}
\address{Department of Computer Science, Mathematics and Physics.\\
Faculty of Science and Technology.\\
The University of the West Indies\\
Cave Hill \\
St.\ Michael \\
Barbados W.I. \\
 }
\email{peter.chami@cavehill.uwi.edu}
\email{bernd.sing@cavehill.uwi.edu}

\author{Norris Sookoo}
\address{The University of Trinidad and Tobago \\
O'Meara Campus \\
Arima\\
Trinidad W.I.}
\email{norris.sookoo@utt.edu.tt }

\date{\today}



\begin{abstract}
We investigate polynomials, called $m$-polynomials, whose generator polynomial has coefficients that can be arranged as a matrix, where $q$ is a positive integer greater than one. Orthogonality relations are established and coefficients are obtained for the expansion of a polynomial in terms of $m$-polynomials. We conclude this article by an implementation in MATHEMATICA of $m$-polynomials and the results obtained for them.
\end{abstract}

\maketitle

\section{Introduction}

Matrices have been the subject of much study, and large bodies of results have
been obtained about them. We study the interplay between the theory of
matrices and the theory of orthogonal polynomials. Interesting results have been
obtained for Krawtchouk polynomials \cite{Kra29} by \cite{Szego}, \cite{Eagleson} and
\cite{Dunkl} (also see the review article \cite{FK05}) and also for generalized Krawtchouk polynomials by
\cite{Williams}. More recently, \cite{Krasikov} obtained conditions for the
existence of integral zeros of binary Krawtchouk polynomials. \cite{Sookoo}
obtained properties for generalized Krawtchouk polynomials. Other
generalizations of binary Krawtchouk polynomials have also been considered.
\cite{Habsieger} generalized some properties of binary Krawtchouk polynomials
to $q$-Krawtchouk polynomials. \cite{Atakis} derived orthogonality relations for
quantum and $q$-Krawtchouk polynomials and showed that affine $q$-Krawtchouk
polynomials are dual to quantum $q$-Krawtchouk polynomials. In this paper, we
define and study a generalizations of Krawtchouk polynomials, namely, $m$-polynomials.

The Krawtchouk polynomial $K_{k}\left(  x\right)  $ is given by
\[
K_{k}\left(  x\right)  =\sum_{j=0}^{k}\left(  -1\right)  ^{j}\binom{x}%
{j}\binom{n-x}{k-j}%
\]
where $n$ is a natural number and $x\in\left\{  0,1,\ldots,n\right\}$. The
generator polynomial is
\[
\left(  1+z\right)  ^{n-x}\left(  1-z\right)  ^{x}=\sum_{k=0}^{n}K_{k}\left(
x\right)  z^{k}%
\]
The generalized Krawtchouk polynomial $K_{p}\left(  s\right)  $ is obtained by
generalizing the above generator polynomial as follows%
\[
\prod\limits_{l=0}^{q-1}\left(  \sum_{i=0}^{q-1}\chi\left(  \overline{\omega}%
_{i}\overline{\omega}_{l}\right)  z_{i}\right)  ^{s_{i}}=\sum_{p\in
V\left(  n,q\right)  }K_{p}\left(  s\right)  z^{p}%
\]
where $\sum s_{i}=\sum p_{i}=n$, $q$ is a prime power, the $z_{i}$ are
indeterminate, the field $GF\left(  q\right)$ with $q$ elements is $\left\{
0,\overline{\omega}_{1},\overline{\omega}_{2},\ldots,\overline{\omega}%
_{q-1}\right\}$, and $\chi$ is a character.

The above information about Krawtchouk polynomials and generalized Krawtchouk
polynomials was taken from \cite{Williams}.

If we replace the $\chi\left(  \overline{\omega}_{i}\overline{\omega}%
_{l}\right)  $ by arbitrary scalars in the last equation, we obtain the generator polynomial of
$m$-polynomials $MG\left(  p,s\right)$, see Definition~\ref{def:2} below. These $m$-polynomials 
are the subject of study in this article.

In Section~\ref{sec:def} we present relevant notations and definitions. In Section~\ref{sec:GP}, we introduce the generator polynomial. The associated matrix of coefficients $G$ can be any square matrix, and so the question that immediately arises is how are the properties of the $m$-polynomials related to the properties of $G$. We will establish that if $G$ is a generalized Hadamard matrix, then the associated $m$-polynomials satisfy orthogonality conditions. We also obtain coefficients for the expansion of a polynomial in terms of $m$-polynomials in Section~\ref{sec:poly}. Finally,  in Section~\ref{sec:math} we implement the results obtained here in MATHEMATICA, so the reader may easily obtain and explore $m$-polynomials for any matrix $G$.

\section{Definitions and Notations}\label{sec:def}

In this article,  $\mathbb{N}_{0}^{}$ denotes $\left\{  0,1,2,3,\ldots\right\}$. We use the convention that if $p\in \mathbb{N}_{0}^{q}$, then $p$ has components $\left(  p_{0},p_{1},\ldots,p_{q-1}\right)$. So, if $r_{\left(  i\right)  }\in \mathbb{N}_{0}^{q}$, then $r_{\left(  i\right)  }$ denotes $\left(r_{i,0},r_{i,1},\ldots,r_{i,q-1}\right)$.

We use the $\ell^1$-norm (the ``taxicab-metric'') to measure the \emph{length} $|p|$ of  $p\in \mathbb{N}_{0}^{q}$, i.e., $\left\vert p\right\vert =\sum\limits_{i=0}^{q-1}p_{i}$. We define the \emph{set of weak compositions of $n$ into $q$ numbers} by  $V\left(  n,q\right)  =\left\{  p\in
\mathbb{N}_{0}^{q}\mathbin:\left\vert p\right\vert =n\right\}$; in other words, $V(n,q)$ the subset of $q$-dimensional nonnegative vectors of length $n$. We note that the set $V(n,q)$ has cardinality $\binom{n+q-1}{q-1}$. For $p\in \mathbb{N}_{0}^{q}$, we use the multi-index notation $p!=\prod\limits_{i=0}^{q-1}p_{i}!$. Similarly, for a variable $z=\left(  z_{0},z_{1},\ldots,z_{q-1}\right)  \in \mathbb{C}^{q}$ and $p\in\mathbb{N}_{0}^{q}$, we write $z^{p}=\prod_{j=0}^{q-1}z_{j}^{p_{i}}$ (where the convention $0^0=1$ is used). We note that the multinomial theorem reads 
\[
\left(\sum\limits_{i=0}^{q-1} z_i\right)^n = \sum\limits_{p\in V(n,q)} \binom{n}{p}\cdot z^{p} =  \sum\limits_{p\in V(n,q)} \frac{n!}{p_0!\cdots p_{q-1}!}\cdot  z_0^{p_0}\cdots z_{q-1}^{p_{q-1}}.
\]

In the following, $G$ denotes an arbitrary $q\times q$ matrix. We use the following convention to refer to the entries of a $q\times q$ matrix \cite{Pdelsarte}: The entry in the $i$th row and $j$th column is called the \emph{$\left(i-1,j-1\right)$ entry} where $ i,j=1,2,\ldots,q$. Thus, the $\left(i,j\right)$ entry of $G$ is denoted by $g_{ij}$ where $i,j=0,1,\ldots,q-1$. Given a matrix $G$, the matrix that remains when the first row and the first column of $G$ are removed is called the \emph{core} of $G$.

The next definition is well known.

\begin{definition}
For an integer $q$ greater than one, a $q\times q$ matrix $H$ is called a \emph{generalized Hadamard matrix} if $H\overline{H}^{T}=qI_{q}$, where $\overline{H}^{T}$ is the complex conjugate transpose of $H$, and $I_{q}$ is the $q\times q$ identity matrix.
\end{definition}

We now define the $m$-polynomial with respect to a matrix $G$.

\begin{definition}\label{def:2}
Let $s=\left(  s_{0},s_{1,},\ldots,s_{q-1}\right)$ and $p=\left(  p_{0},p_{1},\ldots,p_{q-1}\right)  $ be elements of $V\left(  n,q\right)  $. The \emph{$m$-polynomial in $s$ with respect to $G$ having parameter $p$} is denoted $MG\left(  p;s\right) $ or $MG\left(  p_{0},p_{1}, \ldots,p_{q-1};s_{0},s_{1},\ldots,s_{q-1}\right)  $ and given by
\[
MG\left(  p;s\right)  =\sum_{r_{i,j}}\frac{s!} {\prod\limits_{i,j}r_{i,j}!}\prod_{a=0}^{q-1}\prod_{b=0}^{q-1}g_{ab}^{r_{a,b}}
= s!\, \sum_{r_{i,j}}\prod_{a=0}^{q-1}\prod_{b=0}^{q-1}\frac{g_{ab}^{r_{a,b}}}{r_{a,b}!}
\]
where the summation is taken over all sets of non-negative integers $r_{i,j}$ (with $ i,j=0,1,\ldots,q-1$)  satisfying
\[
\sum_{j=0}^{q-1}r_{i,j}=s_{i},\quad i=0,1,\ldots,q-1
\]
and%
\[
\sum_{i=0}^{q-1}r_{i,j}=p_{j},\quad j=0,1,\ldots,q-1.
\]
\end{definition}

\section{The Generator Polynomial \& Orthogonality Relations}\label{sec:GP}

The values of the $m$-polynomial with respect to a matrix $G$ can be derived from a generator polynomial.

\begin{theorem}[Generator polynomial]\label{thm:pol}
Let $s=\left(  s_{0},s_{1},\ldots,s_{q-1}\right)  \in V\left(  n,q\right)  $ and let $G=\left(  g_{ij}\right)$ be a $q\times q$ matrix. Then,
\[
\prod_{i=1}^{q-1}\left(  \sum_{j=0}^{q-1}g_{ij}z_{j}\right)  ^{s_{i}}=\sum_{p\in
V\left(  n,q\right)  }MG\left(  p;s\right)  z^{p}%
\]
where $p=\left(  p_{0},p_{1},\ldots,p_{q-1}\right)$.
\end{theorem}

\begin{proof}%
This is an application of the multinomial theorem (recall that for $r_{\left(  i\right)  } \in \mathbb{N}_{0}^{q}$ we have $r_{\left(  i\right)  }=\left(  r_{i,0},r_{i,1},\ldots,r_{i, q-1  }\right)$):
\begin{eqnarray*}
\prod_{i=0}^{q-1}\left(  \sum_{j=0}^{q-1}g_{ij}z_{j}\right)  ^{s_{i}} 
& = & \prod_{i=0}^{q-1}\left\{  \sum_{r_{\left(  i\right)  }\in V\left(  s_{i},q\right)}\frac{s_{i}!}{r_{\left(  i\right)  }!}\left[  \prod_{j=0}^{q-1}\left(g_{ij}z_{j}\right)  ^{r_{i,j}}\right]  \right\} \\
 & = & \sum_{\substack{r_{\left(  i\right)  }\in V\left(  s_{i},q\right), \\ i=0,1,\ldots,q-1}} \frac{s!}{\prod\limits_{i,j}r_{i,j}!}\left(  \prod_{a=1}^{q-1}\prod_{b=0}^{q-1} g_{ab}^{r_{a,b}}\right)  \prod_{k=0}^{q-1}z_{k}^{r_{0,k}+r_{1,k}+\ldots+r_{q-1,k}}\\
& = & \sum_{p\in V\left(  n,q\right)  }MG\left(  p;s\right)  z^{p}%
\end{eqnarray*}
where $p_{k}=r_{0,k}+r_{1,k}+\ldots+r_{q-1,k}$.
\end{proof}

\begin{remark}
Theorem~\ref{thm:pol} can also be used for summation results of $m$-polynomials over $V(n,q)$: using $z=(1,\ldots,1)$ we obtain
\[
\sum_{p\in V\left(  n,q\right)  }MG\left(  p;s\right) =  \prod_{i=0}^{q-1}\left(  \sum_{j=0}^{q-1}g_{ij}\right)  ^{s_{i}},
\]
i.e., the product of the $s_i$th power of the $i$th column sum of $G$. 
\end{remark}

For a generalized Hadamard matrix $G$, the multinomial theorem yields the following orthogonality relation for the corresponding $m$-polynomials.

\begin{theorem}[Orthogonality relation]\label{thm:OR}
If $G$ is a generalized Hadamard matrix, then the $m$-polynomials $MG\left(
p,s\right)  ,\left(  p,s\in V\left(  n,q\right)  \right)$, satisfy the
orthogonality relations:%
\[
\sum_{s\in V\left(  n,q\right)  }\frac{1}{s!}MG\left(  p;s\right)
\overline{MG\left(  t;s\right)  }=\frac{q^{n}}{p!}\,\delta_{p,t}%
\]
where%
\[
s   =\left(  s_{0},s_{1},\ldots,s_{q-1}\right), \
p  =\left(  p_{0},p_{1},\ldots,p_{q-1}\right) \ \text{and} \ 
t  =\left(  t_{0},t_{1},\ldots,t_{q-1}\right),
\]
and 
\[
\delta_{p,t}=\prod_{i=0}^{q-1}\delta_{p_{i},t_{i}} = \begin{cases} 1 & \text{if}\ p_i=t_i\ \text{for all}\ i,\\ 0 & \text{otherwise},\end{cases}
\]
denotes \emph{Kronecker's delta}.
\end{theorem}

\begin{proof}
Let $z=\left(z_0,z_1,\ldots,z_{q-1}\right),\ y=\left(y_0,y_1,\ldots,y_{q-1}\right)\in\mathbb{C}^q$ and define
\[
J=\sum_{i=0}^{q-1}\left(  \sum_{j=0}^{q-1}g_{ij}z_{j}\right)  \left(\sum_{k=0}^{q-1}\overline{g_{ik}}y_{k}\right)=\sum_{i=0}^{q-1}\sum_{j=0}^{q-1}\sum_{k=0}^{q-1} g_{ij}\overline{g_{ik}}\,z_{j}y_{k}.
\]
Then, on the one hand we have
\begin{equation*}
J=q\sum_{j=0}^{q-1}z_{j}y_{j}
\end{equation*}
since $G$ is a generalized Hadamard matrix. Consequently, 
\[
J^{n}=q^{n}\left\{  \sum_{s\in V\left(  n,q\right)  }\frac{n!}{s!}\left(
z_{0}y_{0}\right)^{s_{0}}\left(  z_{1}y_{1}\right)^{s_{1}}\ldots\left(
z_{q-1}y_{q-1}\right)^{s_{q-1}}\right\}.
\]

On the other hand, the multinomial theorem yields
\begin{eqnarray*}
J^{n} & = &  \sum_{s\in V\left(  n,q\right)  }\frac{n!}{s!} \prod_{i=0}^{q-1} \left(  \sum_{j=0}^{q-1}g_{ij}\,z_{j}\right)^{s_i} \left(\sum_{k=0}^{q-1}\overline{g_{ik}}\,y_{k}\right)^{s_i} \\
& = &  \sum_{s\in V\left(  n,q\right)  }\frac{n!}{s!}\left[  \sum_{p\in
V\left(  n,q\right)  }MG\left(  p;s\right)\,  z^{p}\right] \, \left[  \sum_{t\in
V\left(  n,q\right)  }\overline{MG\left(  t;s\right)  }\,y^{t}\right].
\end{eqnarray*}

Equating coefficients of $z^{p}y^{t}$ in the two above expressions for $J^{n}$, we obtain the desired result.
\end{proof}

If $G$ is a generalized Hadamard matrix and also satisfies certain additional
conditions, then it is possible to establish that the corresponding
$m$-polynomials satisfy additional orthogonality conditions. We use the
following three results in proving this.

\begin{lemma}\label{lem:1}
If $G$ is symmetric, i.e.,\ $G=G^T$, then $MG\left(  s;p\right)  =\frac{p!}{s!}\,MG\left(
p;s\right)$.
\end{lemma}

\begin{proof}
By Definition~\ref{def:2}, we have (where $\sum\limits_{j=0}^{q-1}h_{i,j}=p_{i}$ and $\sum\limits_{i=0}^{q-1}h_{i,j}=s_{j}$) 
\begin{eqnarray*}
MG\left(  s;p\right)  & = & p!\, \sum_{h_{i,j}}\prod_{a=0}^{q-1}\prod_{b=0}^{q-1}\frac{g_{ab}^{h_{a,b}}}{h_{a,b}!}
 \\
 & = & \frac{p!}{s!}\,s!\,\sum_{h_{i,j}}\prod_{a=0}^{q-1}\prod_{b=0}^{q-1}\frac{g_{ba}^{h_{a,b}}}{h_{a,b}!}
\\
 & = & \frac{p!}{s!}\,s!\,\sum_{r_{i,j}}\prod_{b=0}^{q-1}\prod_{a=0}^{q-1}\frac{g_{ba}^{r_{b,a}}}{r_{b,a}!}
\\
& = & \frac{p!}{s!}\,MG\left(  p;s\right),  
\end{eqnarray*}
where we define $r_{i,j}=h_{j,i}$ (and thus $\sum\limits_{j=0}^{q-1}r_{i,j}=s_{i}$ and $\sum\limits_{i=0}^{q-1}r_{i,j}=p_{j}$), and since $g_{ab}=g_{ba}$.
\end{proof}

The following results can be obtained in a similar manner to Lemma~\ref{lem:1}.

\begin{lemma}\label{lem:2}
If $G$ has a symmetric core, $g_{0,j}=1$ for $j=0,\ldots,q-1$ and  $g_{i,0}=-1$ for $i=1,\ldots,q-1$,
then $MG\left(  s;p\right)  =(-1)^{s_0+p_0}\,\frac{p!}{s!}MG\left(  p;s\right)$.
\end{lemma}

Such symmetry relations yield additional orthogonality relations.

\begin{theorem}[Additional orthogonality relation]\label{thm:sym}
Let $G$ be a generalized Hadamard matrix. 
\begin{enumerate}
\item If in addition $G$ is symmetric, then 
\[
\sum_{s\in V\left(  n,q\right)  }MG\left(  p;s\right)  \overline{MG\left(s;t\right)  }=q^{n}\,\delta_{p,t}.
\]
\item If in addition $G$ has a symmetric core, $g_{0,j}=1$ for $j=0,\ldots,q-1$ and  $g_{i,0}=-1$ for $i=1,\ldots,q-1$,
then
\[
\sum_{s\in V\left(  n,q\right)  } (-1)^{s_0}\,MG\left(  p;s\right)  \overline{MG\left(s;t\right)  }=(-1)^{t_0}\,q^{n}\,\delta_{p,t}.
\]
\end{enumerate}

\end{theorem}

\begin{proof}
This follows immediately from Theorem~\ref{thm:OR}, Lemma~\ref{lem:1} and Lemma~\ref{lem:2}.
\end{proof}

\begin{remark}
It turns out that a slight modification of the above $m$-polynomials have already been considered in \cite{Lou65,CL98}. For a matrix $G$ and $p,s\in V(n,q)$, \cite[Eq.~(2.3)]{CL98} defines polynomials $L_{p,s}(G)$ by $L_{p,s}(G)=\frac1{s!}\,MG(p;s)$ in our notation. For these polynomials, the following multiplication property is proved, see \cite[Theorem 3.2]{CL98},
\[
L_{p,t}(G_1\,G_2) =\sum_{s\in V\left(  n,q\right)  } s!\, L_{p,s}(G_1) L_{s,t}(G_2), 
\]
for $q\times q$ matrices $G_1, G_2$. With $G_1=G$, $G_2=\overline{G}^T$ and thus $G_1\,G_2=G\,\overline{G}^T=q\,I_q$, Theorem~\ref{thm:OR} becomes a special case of this equation by noting that $L_{s,t}(G^T)=L_{t,s}(G)$ (see \cite[Eq.~(4.5)]{CL98}) and $L_{p,t}(q\,I_q)=\frac{q^n}{p!}\,\delta_{p,t}$ (using the notation of Definition~\ref{def:2}, the polynomial on the left can only be nonzero if $r_{i,j}=0$ whenever $i\neq j$ and $p_i=r_{i,i}=t_i$ otherwise). 
\end{remark}

\section{Expansion of a Polynomial}\label{sec:poly}

The coefficients for the expansion of a polynomial in terms of Krawtchouk polynomials have been obtained (c.f.\ \cite{Williams}). We obtain a similar result for $m$-polynomials defined in terms of a generalized Hadamard matrix, using orthogonality properties.

\begin{theorem}[Polynomial expansion]\label{thm:poly}
 Let $G$ be a generalized Hadamard matrix and  let $x=\left(  x_{0},x_{1},\ldots,x_{q-1}\right)$ be a variable element of $V\left(  n,q\right)$.
\begin{enumerate}
\item\label{item:1} If $G$ is symmetric and the expansion of a polynomial $\gamma\left(  x\right)$ in terms of $m$-polynomials defined with respect to $G$ is (with coefficients $\alpha_s\in\mathbb{C}$)
\begin{equation}\label{eq:2}
\gamma\left(  x\right)  =\sum_{s\in V\left(  n,q\right)  }\alpha_{s}\cdot MG\left(  x;s\right). 
\end{equation}
Then, for all $\ell\in V\left(  n,q\right)$, 
\[
\alpha_{\ell}=q^{-n}\sum_{i\in V\left(  n,q\right)  }\gamma\left(  i\right)\cdot
\overline{MG\left( \ell,i\right)}.
\]
Similarly, if the expansion of a polynomial $\gamma\left(  x\right)$ in terms of $m$-polynomials is (with coefficients $\beta_s\in\mathbb{C}$)
\begin{equation*}
\gamma\left(  x\right)  =\sum_{s\in V\left(  n,q\right)  }\beta_{s}\cdot MG\left(  s;x\right). 
\end{equation*}
Then, for all $\ell\in V\left(  n,q\right)$, 
\[
\beta_{\ell}=q^{-n}\sum_{i\in V\left(  n,q\right)  }\gamma\left(  i\right)\cdot
\overline{MG\left( i,\ell\right)}.
\]
\item If $G$ has a symmetric core, $g_{0,j}=1$ for $j=0,\ldots,q-1$ and  $g_{i,0}=-1$ for $i=1,\ldots,q-1$, then the corresponding coefficients can be calculated for all $\ell\in V\left(  n,q\right)$ by
\[
\alpha_{\ell}= (-1)^{\ell_0}\, q^{-n}\sum_{i\in V\left(  n,q\right)  } (-1)^{i_0}\,\gamma\left(  i\right)\cdot
\overline{MG\left( \ell,i\right)},
\]
respectively,
\[
\beta_{\ell}= (-1)^{\ell_0}\,q^{-n}\sum_{i\in V\left(  n,q\right)  } (-1)^{i_0}\,\gamma\left(  i\right)\cdot
\overline{MG\left( i,\ell\right)}.
\]
\end{enumerate}
\end{theorem}

\begin{proof}
Let $x=i$. Multiplying both sides of Eq.~\eqref{eq:2} by $\overline{MG\left(  \ell,i\right)  }$ and summing each side over all elements of $V\left(  n,q\right)  $ yields
\begin{eqnarray*}
\sum_{i\in V\left(  n,q\right)  }\gamma\left(  i\right)\,  \overline{MG\left( \ell,i\right)  } 
& = & 
\sum_{i\in V\left(  n,q\right)  }\,\sum_{s\in V\left(n,q\right)  }\alpha_{s}\,\overline{MG\left(\ell,i\right)  }\,MG\left(  i,s\right)\\
& = & \sum_{s\in V\left(  n,q\right)  }\alpha_{s}\,\sum_{i\in V\left(  n,q\right)} \overline{MG\left(  l,i\right)  }\, MG\left(  i,s\right)\\
& = & \sum_{s\in V\left(  n,q\right)  }\alpha_{s}\,\overline{\sum_{i\in V\left(  n,q\right)} {MG\left(  l,i\right)  }\, \overline{MG\left(  i,s\right)} }\\
& \stackrel{(\star)}{=} & \sum_{s\in V\left(  n,q\right)  }\alpha_{s}\cdot\,q^{n}\cdot \delta_{\ell,s}\\
&  = & \alpha_{\ell}\,q^{n},
\end{eqnarray*}
where step $(\star)$ uses Theorem~\ref{thm:sym}; the first result in \ref{item:1} follows. The other results are proved similarly.
\end{proof}

\section{Mathematica Code}\label{sec:math}

Here, we present MATHEMATICA code
to obtain $m$-polynomials.
First, we specify the matrix $G$ and the length $n$:
\begin{verbatim} 
> g = {{1, 1}, {-1, 1}}
> n = 6
\end{verbatim} 
From this we can calculate $q$, the set $V(n,q)$ and initialize the variable $z$:
\begin{verbatim} 
> q = Union[Dimensions[g]][[1]]
> pall = Sort[Flatten[Map[Permutations, 
              IntegerPartitions[n, {q}, Range[0, n]]], 1]]
> z = Table[f[i], {i, q}]
\end{verbatim} 
Using Theorem~\ref{thm:pol}, we obtain the $m$-polynomial via its generator:
\begin{verbatim} 
> generator[s_] := Expand[Product[
          (Sum[g[[i, j]] z[[j]], {j, 1, q}])^(s[[i]]), {i, 1, q}]]
> mg[p_, s_] := Coefficient[generator[s], Inner[Power, z, p, Times]]
\end{verbatim} 
The values of the $m$-polynomial can be shown using the following command:
\begin{verbatim} 
> TableForm[Table[mg[pall[[i]], pall[[j]]], {i, Length[pall]}, 
            {j, Length[pall]}], TableHeadings -> {pall, pall}, 
             TableDepth -> 2]
\end{verbatim} 
For the above chosen matrix $G$ and length $n$, the output for $MG(p;s)$ looks as follows ($p$ denotes the row and $s$ the column, e.g., $MG((2,4);(3,3))=-3$):
\[
\begin{array}{r|ccccccc}
& \{0,6\} & \{1,5\} & \{2,4\} & \{3,3\} & \{4,2\} & \{5,1\} & \{6,0\} \\ \hline
\{0,6\} & 1 & -1 & 1 & -1 & 1 & -1 & 1 \\
\{1,5\} &  -6 & 4 & -2 & 0 & 2 & -4 & 6\\
\{2,4\} &  15 & -5 & -1 & 3 & -1 & -5 & 15\\
\{3,3\} & -20 & 0& 4& 0& -4& 0& 20 \\
\{4,2\} & 15&	5&	-1&	-3&	-1&	5&	15 \\
\{5,1\} & -6	& -4&	-2&	0&	2&	4&	6 \\
\{6,0\} & 1 & 1&	1&	1&	1&	1&	1 \\
\end{array}
\]
Thus, we are able to evaluate the $m$-polynomial $MG(p;s)$ for any $p,s\in V(n,q)$. In this case, we can express the $m$-polynomials as univariant polynomials in $(s_0-s_1)$ as follows\footnote
{
	We obtained these with the help of the MATHEMATICA function \texttt{Fit}.
} 
(we also recall that $|(s_0,s_1)|=s_0+s_1=6$ here):
\begin{eqnarray*}
MG\left((0,6);(s_0,s_1)\right) & = & \frac1{720}\,(s_0-s_1)^6-\frac5{72}\,(s_0-s_1)^4+\frac{34}{45}\,(s_0-s_1)^2-1  \\
MG\left((1,5);(s_0,s_1)\right) & = & \frac1{120}\,(s_0-s_1)^5-\frac13\,(s_0-s_1)^3+\frac{11}5\,(s_0-s_1)  \\
MG\left((2,4);(s_0,s_1)\right) & = &  \frac1{24}\,(s_0-s_1)^4-\frac76\,(s_0-s_1)^2+3 \\
MG\left((3,3);(s_0,s_1)\right) & = & \frac16\,(s_0-s_1)^3-\frac83\,(s_0-s_1) \\
MG\left((4,2);(s_0,s_1)\right) & = & \frac12\,(s_0-s_1)^2-3\\
MG\left((5,1);(s_0,s_1)\right) & = & (s_0-s_1)\\
MG\left((6,0);(s_0,s_1)\right) & = & 1 \\
\end{eqnarray*}
Similarly, we may express the $m$-polynomials as univariant polynomials in $(p_0-p_1)$ as follows:
\begin{eqnarray*}
MG\left((p_0,p_1);(0,6)\right) & = &\frac{77}{3840}\,(p_0-p_1)^6-\frac{203}{192}\,(p_0-p_1)^4+\frac{1519}{120}\,(p_0-p_1)^2-20\\
MG\left((p_0,p_1);(1,5)\right) & = & \frac7{640}\,(p_0-p_1)^5-\frac{49}{96}\,(p_0-p_1)^3+\frac{131}{30}\,(p_0-p_1) \\
MG\left((p_0,p_1);(2,4)\right) & = & -\frac7{3840}\,(p_0-p_1)^6+\frac{7}{64}\,(p_0-p_1)^4-\frac{199}{120}\,(p_0-p_1)^2+4\\
MG\left((p_0,p_1);(3,3)\right) & = & -\frac7{1920}\,(p_0-p_1)^5+\frac{19}{96}\,(p_0-p_1)^3-\frac{67}{30}\,(p_0-p_1) \\
MG\left((p_0,p_1);(4,2)\right) & = &  \frac7{11520}\,(p_0-p_1)^6-\frac{25}{576}\,(p_0-p_1)^4+\frac{329}{360}\,(p_0-p_1)^2-4\\
MG\left((p_0,p_1);(5,1)\right) & = &  \frac1{384}\,(p_0-p_1)^5-\frac{17}{96}\,(p_0-p_1)^3+\frac{19}6\,(p_0-p_1) \\
MG\left((p_0,p_1);(6,0)\right) & = & -\frac1{2304}\,(p_0-p_1)^6+\frac{23}{576}\,(p_0-p_1)^4-\frac{101}{72}\,(p_0-p_1)^2+20  \\
\end{eqnarray*}
By Theorem~\ref{thm:poly}, we can express any polynomial $\gamma(x)=\gamma(x_0,\ldots,x_{q-1})$ in terms of $m$-polynomials. Using MATHEMATICA to find the corresponding coefficients $\alpha_s$ and $\beta_s$ in the expansion in terms of $m$-polynomials, e.g., for the polynomial $p(x_0,x_1)=x_0\,x_1$, is achieved as follows:
\begin{verbatim}
> p[{x0_, x1_}] := x0 x1
> Table[1/q^n Total[Table[p[pall[[i]]] 
          Conjugate[mg[pall[[j]], pall[[i]]]], {i, Length[pall]}]],
           {j, Length[pall]}]
> Table[1/q^n Total[Table[p[pall[[i]]] 
          Conjugate[mg[pall[[i]], pall[[j]]]], {i, Length[pall]}]],
           {j, Length[pall]}]
\end{verbatim}
The output of the later two lines is
\[
\left\{-\frac{35}{64},0,-\frac{7}{64},0,\frac{39}{64},0,\frac{3}{64}\right\}\qquad\text{and}\qquad \left\{0,0,0,0,-\frac{1}{2},0,\frac{15}{2}\right\},
\]
giving the list of coefficients $\alpha_s$ respectively $\beta_s$ for $s\in V(n,q)$. I.e., the expansion in terms of $m$-polynomials here reads:
\begin{eqnarray*}
x_0\,x_1 & = & -\frac{35}{64}\,MG\left((x_0,x_1);(0,6)\right)-\frac{7}{64}\,MG\left((x_0,x_1);(2,4)\right)\\ & & \quad +\frac{39}{64}\,MG\left((x_0,x_1);(4,2)\right)+\frac{3}{64}\,MG\left((x_0,x_1);(6,0)\right) \\
x_0\,x_1 & = & -\frac12\,MG\left((4,2);(x_0,x_1)\right)+\frac{15}2\,MG\left((6,0);(x_0,x_1)\right) \\
\end{eqnarray*}
As a check, we have (recall that $x_0+x_1=6$)
\begin{multline*}
 -\frac12\,MG\left((4,2);(x_0,x_1)\right)+\frac{15}2\,MG\left((6,0);(x_0,x_1)\right)\\ =\ -\frac12\cdot\left(\frac12(x_0-x_1)^2-3\right)+\frac{15}2\cdot 1 \ =\ x_0\,(6-x_0)\ =\ x_0\,x_1.
\end{multline*}

For a generalized Hadamard matrix with symmetric core, $g_{0,j}=1$ for $j=0,\ldots,q-1$ and  $g_{i,0}=-1$ for $i=1,\ldots,q-1$, e.g.,
\[
G=\begin{pmatrix} 1 & 1 & 1 \\ -1 & \frac{1-\sqrt{3}}2 &  \frac{1+\sqrt{3}}2 \\  -1 & \frac{1+\sqrt{3}}2 &  \frac{1-\sqrt{3}}2 \end{pmatrix},
\]
the above MATHEMATICA code using Theorem~\ref{thm:poly} has to be modified as follows:
\begin{verbatim}
> Table[1/q^n Total[Table[(-1)^(pall[[i, 1]] + pall[[j, 1]])  
          p[pall[[i]]] Conjugate[mg[pall[[j]], pall[[i]]]], 
           {i, Length[pall]}]], {j, Length[pall]}]
> Table[1/q^n Total[Table[(-1)^(pall[[i, 1]] + pall[[j, 1]]) 
          p[pall[[i]]] Conjugate[mg[pall[[i]], pall[[j]]]], 
           {i, Length[pall]}]], {j, Length[pall]}]
\end{verbatim}

\section{Outlook}

Krawtchouk polynomials and their generalisation appear in many areas of mathematics, see\footnote
{
	A website called the ``Krawtchouk Polynomials Home Page'' by V.~Zelenkov at \texttt{http://orthpol.narod.ru/eng/index.html} collecting material about M.~Krawtchouk and the polynomials that bear his name has, unfortunately, not been updated in a while.
} 
\cite{FK05}: harmonic analysis \cite{Kra29,Dunkl,Szego}, statistics \cite{Eagleson}, combinatorics and coding theory \cite{Pdelsarte,Habsieger,Krasikov,Lev95,Williams}, probability theory \cite{FK05}, representation theory (e.g., of quantum groups) \cite{Atakis,CL98,Lou65}, difference equations \cite{BFSVZ}, and pattern recognition \cite{SD12}. However, our motivation in this article was primarily driven by generalising the results obtained in \cite{Sookoo} -- and thus shedding more light on the qualitative structure of (generalisations of) Krawtchouk polynomials -- and not yet with a specific application in mind. By providing the MATHEMATICA code to obtain $m$-polynomials, we also hope that other researchers are encouraged to explore them and see if they can be used in their research.


\end{document}